\documentclass[a4paper]{amsart}

\usepackage{amsfonts}
\usepackage{amsthm}
\usepackage{amstext}
\usepackage{amsmath}
\usepackage{amscd}
\usepackage{amssymb}
\usepackage[mathscr]{eucal}
\usepackage{epsf,epsfig}

\newcommand{\ssq}{\ensuremath{\subseteq}}
\newcommand{\smin}{\ensuremath{\setminus}}
\newcommand{\eps}{\ensuremath{\varepsilon}}

\newcommand{\T}{\ensuremath{\mathbb{T}}}

\newcommand{\N}{\ensuremath{\mathbb{N}}} 
\newcommand{\R}{\ensuremath{\mathbb{R}}}
\newcommand{\Z}{\ensuremath{\mathbb{Z}}}

\newcommand{\A}{\ensuremath{\mathbb{A}}}

\newcommand{\Id}{\ensuremath{\mathrm{Id}}}

\newcommand{\Leb}{\ensuremath{\mathrm{Leb}}}

\newcommand{\inte}{\ensuremath{\mathrm{int}}}

\newcommand{\kreis}{\ensuremath{\mathbb{T}^{1}}}

\newcommand{\sltr}{\ensuremath{\textrm{SL}(2,\mathbb{R})}}
\newcommand{\torus}{\ensuremath{\mathbb{T}^2}}

\newcommand{\alphlist}{\begin{list}{(\alph{enumi})}{\usecounter{enumi}\setlength{\parsep}{2pt}
      \setlength{\itemsep}{1pt} \setlength{\topsep}{5pt}
      \setlength{\partopsep}{3pt}}}

\newcommand{\arablist}{\begin{list}{(\arabic{enumi})}{\usecounter{enumi}\setlength{\parsep}{2pt}
          \setlength{\itemsep}{1pt} \setlength{\topsep}{5pt}
          \setlength{\partopsep}{3pt}}}

\newcommand{\romanlist}{\begin{list}{(\roman{enumi})}{\usecounter{enumi}\setlength{\parsep}{2pt}
              \setlength{\itemsep}{1pt} \setlength{\topsep}{5pt}
              \setlength{\partopsep}{3pt}}}

 \newcommand{\listend}{\end{list}}

\newcommand{\bulletlist}{\begin{list}{$\bullet$}{\setlength{\parsep}{2pt}
                \setlength{\itemsep}{1pt} \setlength{\topsep}{5pt}
                \setlength{\partopsep}{3pt}\setlength{\leftmargin}{15pt}}}

\newcommand{\kfolge}[1]{\ensuremath{(#1)_{k\in\mathbb{N}}}}

\newcommand{\ncap}{\ensuremath{\bigcap_{n\in\N}}}

\newcommand{\kLim}{\ensuremath{\lim_{k\rightarrow\infty}}}

\newcommand{\halb}{\ensuremath{\frac{1}{2}}}

\theoremstyle{plain}
\newtheorem{theorem}{Theorem}[section]
\newtheorem{proposition}[theorem]{Proposition}
\newtheorem{corollary}[theorem]{Corollary}
\newtheorem{lemma}[theorem]{Lemma}
\theoremstyle{definition}

\newtheorem{remark}[theorem]{Remark}

\numberwithin{equation}{section}

\newcommand{\NNN}{\mathbb N}
\newcommand{\ZZZ}{\mathbb Z}

\newcommand{\III}{\mathbb I}
\newcommand{\TTT}{\mathbb T}
\newcommand{\AAA}{\mathbb A}

\newcommand{\dist}{\mathop{\rm dist}}

\newcommand{\inter}{\mathop{\rm int}}

\newcommand{\Orb}{\mathop{\rm Orb}}

\newcommand{\sep}{\ \vert \ }
\newcommand{\ssep}{\, \vert \, }

\newcommand{\comment}[1]{}

\begin{document}

\title{A construction of almost automorphic minimal sets}

\date{\today}

\author[R. Hric]{Roman Hric}
\address{Faculty of Natural Sciences, Matej Bel University, Bansk\'a Bystrica, Slovakia}
\email{hric@savbb.sk, roman.hric@umb.sk}

\author[T. J\"ager]{Tobias J\"ager} \address{Department of
  Mathematics, TU Dresden, Germany}
\email{Tobias.Oertel-Jaeger@tu-dresden.de}



\begin{abstract} We describe a general procedure to construct
  topological extensions of given skew product maps with
  one-dimensional fibres by blowing up a countable number of single
  points to vertical segments. This allows to produce various examples
  of unusual dynamics, including almost automorphic minimal sets of
  almost periodically forced systems, point-distal but non-distal
  homeomorpisms of the torus (as first constructed by Rees) or minimal
  sets of quasiperiodically forced interval maps which are not
  filled-in.
\end{abstract}

\maketitle

\section{Introduction}

Minimal sets and minimal transformations can be considered as the
smallest building blocks of a dynamical system, and consequently their
study has a long tradition in topological dynamics
\cite{ellis1969lectures,auslander1988minimal,kolyada/snoha:2009}. An
important subclass of minimal transformations are the {\em almost
periodic} ones, which can be defined by being conjugate to a minimal
isometry and present the most regular type of minimal dynamics \cite{shen/yi:1998}.
An equivalent definition is to require that the family of iterates is
equicontinuous.
A natural generalisation is given by the notion of almost automorphy.
While its original definition has a rather technical flavour, Veech Structure
Theorem \cite{veech1965almost} allows to state it in a very conceptual way: A map $f$
is {\em almost automorphic} if and only if it is semiconjugate to an almost periodic
map $g$ by an almost 1-1 factor map.

Probably the simplest examples of almost automorphic, but not almost
periodic minimal sets occur for Denjoy examples on the circle
$\kreis=\R/\Z$.
We say a circle homeomorphisms $f$ is of {\em Denjoy type} if its rotation number is
irrational and it exhibits a {\em wandering open interval}.
In this case, there exists a unique minimal set $M\ssq \kreis$, which
is a Cantor set. Further, there is an order-preserving semiconjugacy
$h$ from $f$ to the irrational rotation which collapses the intervals
in $\kreis\smin M$, {\em gaps}, by sending them to single points,
while being injective elsewhere.  Hence, as the only points in $M$
with common image are the endpoints of gaps, the map $h$ is almost
1-1 and the irrational rotation is an almost periodic almost 1-1
factor.  Apart from these basic examples, almost automorphic dynamics
often occurs in dynamical systems of intermediate complexity. In
particular, a very rich class of examples is obtained from hull
constructions for quasicrystals or aperiodic tilings.

The focus of this article lies on almost automorphic minimal sets
which occur in continuous skew product transformations of the form
\begin{equation}\label{e.skewproducts}
  f \ : \  \Theta\times X \to \Theta\times X \quad , 
  \quad f(\theta,x)\ = \ (\alpha(\theta),f_\theta(x))
\end{equation}
with an almost periodic homeomorphism $\alpha$ on the base. We say $f$ is an
{\em $\alpha$-forced increasing interval map} if $X\ssq \R$ is an interval and all
fibre maps $f_\theta$ are strictly monotonically increasing.
When $\alpha$ is just an irrational rotation, we speak of {\em quasiperiodic forcing}. 

Given $A\ssq\Theta\times X$, we let $A_\theta=\{x\in X\mid (\theta,x)\in A\}$
and say that $A$ is {\em pinched} if there exists $\theta\in\Theta$ with $\#
A_\theta=1$. Then it is easy to see that any pinched minimal set is almost
automorphic, since in this case an almost 1-1 factor map to $(\Theta,\alpha)$ is
simply given by the projection $\pi_1:\Theta\times X \to \Theta$ to the first
coordinate.  For almost periodically forced increasing interval maps, the converse
is true as well. In fact, in this case every minimal set is pinched, and hence
also almost automorphic \cite{stark:2003,shen/yi:1998}. Note that for more
general almost periodically forced maps the situation is quite different, and
this can already be seen by looking at forced circle homeomorphisms. For
example, the direct product of an irrational rotation and a Denjoy homeomorphism
$f_0$ on the circle has a unique minimal set $M$ (assuming incommensurability of
the rotation numbers), which is the product of $\kreis$ with the minimal Cantor
set $\widehat M$ of $f_0$.  Hence, this set is not pinched -- $M_\theta=\widehat
M$ is uncountable for all $\theta$ -- but it is almost automorphic since it has
the irrational translation on $\torus$ as an almost 1-1 factor.

Almost automorphic minimal sets of pinched type were first observed by
Million{\u{s}\u{c}}ikov \cite{millionscikov:1969} and Vinograd
\cite{vinograd:1975} in certain linear cocycles over almost periodic
base flows. They occur frequently in natural parameter families of
real-analytic skew products over irrational rotations, as later shown
by Herman \cite{herman:1983} for \sltr-cocycles and by
Fuhrmann \cite{fuhrmann:2013} in a more general setting
(see also \cite{bjerkloev:2005,bjerkloev:2007}).  However, all these
constructions make it difficult to produce examples with prescribed
additional properties. Since there are quite a few open problems
concerning the structure of almost automorphic minimal sets, such a
freedom in the construction would be highly desirable.

Our aim here is to describe a general blow-up procedure which allows
to create almost automorphic minimal sets in skew product systems,
starting from almost periodic ones. We say $\Gamma\ssq\Theta\times X$
is an $f$-{\em invariant curve} if there exists a continuous function
$\gamma:\Theta\to X$ such that 
$\Gamma=\{(\theta,\gamma(\theta))\mid\theta\in\Theta\}$.
Note that in this case $\pi_1$ conjugates
$f_{|\Gamma}$ to $\alpha$. The main result is the following.

\begin{theorem}\label{t.main}
  Let $\alpha$ be an almost periodic minimal homeomorphism of an
  infinite compact metric space $\Theta$, and $f:\Theta \times \III
  \to \Theta \times \III$ an $\alpha$-forced increasing interval map.
  Assume that $\Gamma=\{(\theta,\gamma(\theta))\mid\theta\in\Theta\}$
  is an $f$-invariant curve with $\gamma:\Theta \to \inter \III$. 
  Then there exists
  a topological extension $\hat f:\Theta \times \III \to \Theta \times
  \III$ of $f$ with the factor map $h$ from $\hat f$ to $f$,
  $h(\theta,x)=(\theta,h_\theta(x))$, such that the following holds:
  \romanlist
  \item $\hat f$ is an $\alpha$-forced increasing interval map;
  \item all the fibre maps $h_\theta$ are non-decreasing;
  \item $h$ is injective on the complement of $h^{-1}(\Gamma)$;
  \item for a countable number of points in $\Gamma$ the preimage
    under $h$ is a vertical segment, for all other points in $\Gamma$
    it is a singleton;
  \item $h^{-1}(\Gamma)$ does not contain any graph of a continuous
    curve $\eta: \Theta \to \III$;
  \item $h^{-1}(\Gamma)$ is pinched and contains an almost
    automorphic minimal set which is not almost periodic.
  \listend
\end{theorem}

The proof is based on a similar but more technical construction that
was used in \cite{beguin/crovisier/jaeger/leroux:2009} to produce
different examples of transitive but non-minimal skew product
transformations. We hope that the method will turn out useful in order
to address further related problems as well. For this reason, we try
to present the main idea as clearly as possible and develop it in
several steps.  In Section~\ref{Denjoy} we first show how a simple
version of the construction can be used to produce Denjoy
homeomorphisms of the circle. In Section \ref{BlowUp} we then give the
proof of Theorem~\ref{t.main} for the case of increasing interval maps
forced by irrational rotations.  The modifications needed to treat
more general forcing in order to prove Theorem~\ref{t.main} are then
discussed in Section~\ref{General}.

One of the most obvious questions concerning the structure of pinched almost
automorphic minimal sets is the following. Given a minimal set $M$ of a
minimally forced increasing interval map $f$, let
\begin{equation}
  \label{eq:1}
  \varphi^-_M(\theta)\  =\ \inf M_\theta \quad\textrm{ and } 
\quad \varphi^+_M(\theta)\ =\ \sup M_\theta \ .
\end{equation}
Then it is a direct consequence of monotonicity that
\begin{equation}
  \label{eq:2}
  \left[\varphi^-_M,\varphi^+_M\right]\ = \ \left\{(\theta,x)\in\Theta\times X\mid x\in
    \left[\varphi^-_M(\theta),\varphi^+_M(\theta)\right]\right\}
\end{equation}
is $f$-invariant, and $\left[\varphi^-_M,\varphi^+_M\right]$ is
pinched since $M$ is. The question is whether
$M=\left[\varphi^-_M,\varphi^+_M\right]$? If the answer is yes, $M$ is
called {\em filled-in}.  This problem was first observed by Herman in
\cite{herman:1983} (see also \cite{stark:2003}). Examples where the
answer is positive were given by Bjerkl\"ov \cite{bjerkloev:2005}, but
counterexamples have not been described yet. However, in the proof of
Theorem~\ref{t.main} it is not difficult to ensure that the interiors
of the vertical segments appearing in (iv) do not belong to the
minimal set. Hence, we obtain the following.

\begin{proposition}
  There exist almost automorphic minimal sets of quasiperiodically
  forced increasing interval maps which are not filled-in.
\end{proposition}

The proof is again given in Section~\ref{BlowUp}. We also note that
continuous-time examples with analogous properties can be obtained as
suspension flows of discrete-time maps. 

Finally, we give two further applications to demonstrate the
flexibility of the construction. In Section~\ref{Sharkovski} we
discuss some consequences of Theorem~\ref{t.main} in the light of the
Sharkovsky Theorem for minimally forced interval maps. In particular,
we provide an example which shows that periodic orbits of unforced
interval maps cannot simply be replaced by periodic continuous curves
in the forced setting, and hence more sophisticated concepts have to
be used as in \cite{fabbri/jaeger/johnson/keller:2005}. In
Section~\ref{Rees}, we reproduce some examples of point-distal but
non-distal torus homeomorphisms due to Rees \cite{rees:1979}. We say a
homeomorphism $f$ of a compact metric space $(X,d)$ is {\em distal} if
$\inf_{n\in\N} d(f^n(x),f^n(y))>0$ for all $x \neq y\in X$. A point
$x\in X$ is called {\em distal} if $\inf_{n\in\N} d(f^n(x),f^n(y))>0$
for all $y\in X$, and $f$ is called {\em point-distal} if there exists
a distal point.

\begin{proposition}[Rees, \cite{rees:1979}] There exist point-distal
  but non-distal almost automorphic minimal homeomorphisms of the two-torus.
\end{proposition}
\bigskip

\noindent{\bf Acknowledgments.}
Roman Hric was supported by Marie Curie Reintegration Grant
PERG04-GA-2008-239328, by the Slovak Research and Development Agency,
grant APVV-0134-10 and by VEGA, grant 1/0978/11.  Tobias J\"ager and
two visits of Roman Hric to TU Dresden were supported by an
Emmy-Noether-grant of the German Research Council (DFG-grant JA
1721/2-1).  Further, this work is related to the activities of the
Scientific Network ``Skew product dynamics and multifractal analysis''
(DFG grant OE 538/3-1).

\section{Preliminaries}\label{Preliminaries}

For a homeomorphism $f$ of a topological space $X$ we define the {\em
  orbit} of a point $x \in X$ to be the set $\Orb_f(x) := \{ f^n(x)
\ssep n\in\Z\}$ (we omit the subscript for a fixed $f$). We also put
$x_n:=f^n(x)$ for $n\in\Z$.  A subset $A\ssq\R$ is called {\em
  syndetic} if the connected components of $\R\smin A$ are uniformly
bounded in length. Given a homeomorphism $f$ of a compact metric space
$(X,d)$, a point $x\in X$ is {\em almost periodic} (sometimes called
{\em uniformly} or {\em syndetically recurrent}) if for all
neighbourhoods $U$ of $x$ the set $N(x,U)=\{n\in\Z\mid f^n(x)\in U\}$
is syndetic. A homeomorphism $f$ is called {\em pointwise almost
  periodic} if every point is almost periodic, and $f$ is called {\em
  almost periodic} if the set $N(f,\eps)=\{n\in\Z \mid
d(f^n,\Id_X)<\eps\}$ is syndetic for all $\eps>0$. As mentioned
before, this is equivalent to the equicontinuity of $\{f^n \sep
n\in\Z\}$.  Further, in this case $f$ is an isometry with respect to
the metric $\hat d$ given by $\hat d(x,y):=\sup_{n\in\Z}
d(f^n(x),f^n(y))$, which is equivalent to $d$. Consequently, $\Id_X :
(X,d)\to(X,\hat d)$ conjugates $f$ to an isometry
\cite{ellis1969lectures,auslander1988minimal}.

A homeomorphism $f$ is called {\em almost automorphic} if there exists
a point $x\in X$ such that whenever the limit $\tilde x=\kLim
f^{n_k}(x)$ exists for some sequence $\kfolge{n_k}$ of integers, then
$x=\kLim f^{-n_k}(\tilde x)$. As mentioned above, the Veech Structure
Theorem \cite{veech1965almost} states that $f$ is almost automorphic
if and only if it has an almost 1-1 factor which is almost periodic.
Here, we say $(Y,g)$ is a {\em factor} of $(X,f)$ if there exist a
continuous surjection $h: X\to Y$, called a {\em semiconjugacy} or
{\em factor map}, such that $h\circ f=g\circ h$. Conversely, $(X,f)$
is called a {\em (topological) extension} of $(Y,g)$ in this case. The map $h$ is
called {\em almost 1-1} if the set of points $y\in Y$ with exactly one preimage
under $h$ is dense in $Y$, and $(Y,g)$ is called an {\em almost 1-1 factor} in this
case. A minimal set $M$ is called {\em almost periodic} ({\em almost
  automorphic}) if $f_{|M}$ is almost periodic (almost
  automorphic).  \smallskip

Now, let $f$ be an $\alpha$-forced increasing interval map.
An $f$-{\em invariant graph} is an invariant
set $\Gamma$ of the form $\Gamma=\{(\theta,\gamma(\theta)\mid
\theta\in\Theta\}$, where $\gamma:\Theta\to \III$ is a Borel measurable
function. Note that the invariance of $\Gamma$ implies that
\begin{equation}
  \label{eq:4}
  f_\theta(\gamma(\theta)) \ = \ \gamma(\alpha(\theta)) 
  \quad \textrm{ for all } \theta\in\Theta \ .
\end{equation}
If $\gamma$ is continuous, then $\Gamma$ is an $f$-invariant curve in
the sense defined above. It is known that if $\alpha$ is almost
periodic minimal, then the $f$-invariant curves are exactly the almost
periodic minimal sets of $f$
\cite{ellis1969lectures,auslander1988minimal}.

\section{The basic construction: Denjoy's examples revisited} \label{Denjoy}

Denjoy homeomorphisms of the circle are usually constructed by
`blowing up' one (or at most countably many) orbits of an irrational
rotation to wandering intervals. The resulting system is then
semiconjugate to the original rotation, and the corresponding factor
map collapses exactly the blown up intervals while being injective
elsewhere. Depending on the way this basic idea is formalised, the
flavour of this construction can be more combinatorial, topological or
analytic.

Starting point for our construction is an implementation with a strong
measure-theoretic accent. In the case of Denjoy examples, we start by
defining a measure which has an atomic part, supported by the orbit we
want to blow up, and a Lebesgue part to give the measure full
topological support. More precisely, we fix an irrational rotation
$R:\kreis\to\kreis$ and an arbitrary starting point $x_0$ and let
$x_n=R^n(x_0)$ for all $n\in\Z$. Given a sequence of strictly
positive weights $(a_n)_{n\in\Z}$ with $\sum_{n\in\Z} a_n<1$, we let
$b:=1-\sum_{n\in\Z}a_n>0$ and define
\begin{equation}
\nu \  := \ \sum_{n \in \Z} a_n \delta_{x_n} + b\, \Leb_{\T^1}
\end{equation}
where $\delta_x$ denotes the Dirac measure at $x$. Now, instead of
directly constructing the Denjoy homeomorphism $f$, we first use $\nu$
to construct the semiconjugacy $h$ which we use later to define $f$ as
an extension of $R$.  As mentioned, in order to do so it should
collapse a sequence of intervals (which will later be wandering) and
map them to the points $x_n$ of the blow-up orbit. It turns out that
we can define $h$ to be the quantile function of $\nu$, which is the
left-inverse of the distribution function (see
Figure~\ref{f.quantile}). More precisely, we assume without loss of
generality that $0\notin\{x_n \ssep n\in\Z\}$ and let
\begin{equation}
  h(x) \ := \ \min\{y\in \kreis \sep \nu[0,y]\geq x\} \ ,
\end{equation}
where we identify $\kreis$ with $[0,1)$ for taking the minimum. Note
that if we let $g(y):=\nu[0,y]$, then due to the Lebesgue component of
$\nu$ the distribution function $g$ is strictly increasing with
discontinuities of size $a_n$ at the atoms $x_n$, and we have $h\circ
g = \Id_{\kreis}$. The details are given in Lemma~\ref{l.basic} below.
For the use in the later sections, it also includes the case where
$\nu$ is non-atomic.
\begin{center}
\begin{figure}
  \epsfig{file=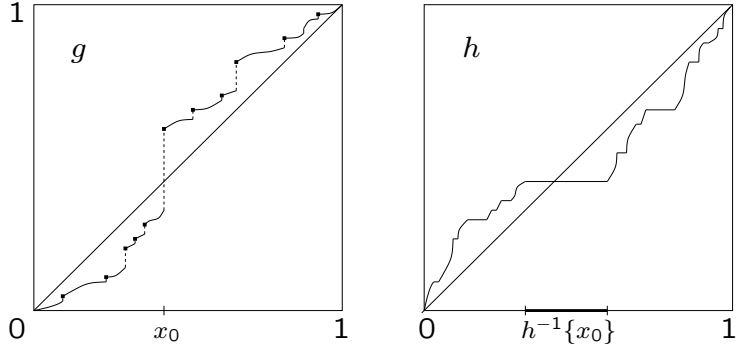, clip=, width=0.8\textwidth}
  \caption{\small The distribution function $g$ on the left and its
    left-inverse, the quantile function $h$, on the right. Atoms of
    $\mu$ correspond to discontinuities of $g$ and to plateaus of
    $h$. \label{f.quantile}}
\end{figure}
\end{center}

If we now let $I_n:=h^{-1}\{x_n\} =: [c_n,d_n]$, then $h$ is injective
on the set
\begin{equation}
  A \ := \ \kreis \smin \bigcup_{n\in\Z} [c_n,d_n) \ .
\end{equation}
In fact, $h$ is even a bijection from $A$ to $\kreis$, and we can
therefore define a map $\tilde f$ on $A$ by
\begin{equation} \label{e.denjoy-def}
 \tilde f\ := \ h_{|A}^{-1}\circ R \circ h_{|A}  \ .
\end{equation}
As shown in Lemma~\ref{l.denjoy-continuity} below, this map $\tilde f$ is
uniformly continuous and strictly order-preserving on $A$. Further, we have
\begin{equation} \label{e.cn_to_cn+1}
\tilde f(c_n)\ = \ h_{|A}^{-1}\circ R \circ
  h_{|A}(c_n) = h_{|A}^{-1}\circ R (x_n) = h_{|A}^{-1}(x_{n+1}) \ = \ c_{n+1}
\end{equation}
for all $n\in\Z$. We can therefore extend $\tilde f$ to a continuous
map $f:\kreis\to\kreis$ by sending each interval $I_n$ to $I_{n+1}$ in
a continuous and monotone way. As we will see below, this yields a
homeomorphism of the circle, and since the intervals $I_n$ are
wandering by construction, $f$ is the desired Denjoy example.
\smallskip

We prove the above claims in the following three lemmas, which will
also be useful in the later sections. Suppose $X$ is a topological
space, equipped with a Borel probability measure $\nu$. The {\em
  topological support} of $\nu$ is defined as
\[
\textrm{supp}(\nu)  \ := \ \left\{ x\in X \sep \nu(U) > 0 \
\textrm{ for any open neighbourhood } U \textrm{ of } x \right\} \ .
\]
If $\nu\{x\}>0$ we call $x$ an {\em atom} and say $\nu$ is {\em
  atomic} in $x$. We call $\nu$ {\em non-atomic} if it has no atoms.
Given a measurable map $h: X \to X$ we denote the {\em push-forward}
of $\nu$ by $h^*\nu := \nu \circ h^{-1}$, that is $h^*\nu(A) :=
\nu(h^{-1}A)$.

\begin{lemma} \label{l.basic} Let $X=\III$ or $\kreis$ and suppose
  $\nu$ is a probability measure on $X$.  Then
\begin{equation}\label{e.h-def}
h(x) := \min \{ y \in X \sep \nu[0,y] \geq x \}
\end{equation}
satisfies $h^* \Leb_X= \nu$.  Moreover \romanlist
\item if $\textrm{supp}(\nu)=X$ then $h$ is continuous and surjective.
  In this case $h^{-1}\{y\} = [\nu[0,y),\nu[0,y]]$;
\item if $\nu$ is non-atomic then $h$ is injective.  \listend It
  follows from (i) and (ii) that when $\nu$ is both non-atomic and has
  full topological support, then $h$ is a homeomorphism with inverse
  $h^{-1}(y) = \nu[0,y]$.
\end{lemma}
Here, if $X=\kreis$ we again identify $\kreis$ with $[0,1)$ for taking
the minimum in (\ref{e.h-def}). Note that this minimum exists due to
continuity from the right of $y \mapsto \nu[0,y]$. 
\proof Suppose $h$
is defined by (\ref{e.h-def}). Then for any $y \in X$ we have
\[
h^{-1}[0,y] \  = \  \{x \in X \sep h(x) \in [0,y] \}
\ = \ \{x \in \III \sep \nu[0,y] \geq x \}
\ = \  [0, \nu[0,y]] \ .
\]
Consequently $\Leb_X(h^{-1}[0,y]) = \nu[0,y]$ and hence $h^* \Leb_X = \nu$.  In
order to show (i), assume that $\textrm{supp}(\nu)=X$. Then
\begin{eqnarray*}
h^{-1}[0,y) & = & \{ x \in X \sep h(x) < y \} \\
& = &  \{ x \in X \sep \exists z < y : x \leq \nu[0,z] \} \\
& = &  \{ x \in X \sep x < \nu[0,y) \}
\ = \  [0, \nu[0,y)) \ .
\end{eqnarray*}
Analogously we get $h^{-1}(y,1] = (\nu[0,y],1]$, and this yields the formula for
$h^{-1}\{y\}$. Furthermore, we see that preimages of open intervals are open,
and hence $h$ is continuous. Finally, if $\nu$ in non-atomic then for every
$y\in X$ we obtain $h^{-1}\{y\} = \{\nu[0,y]\}$, so $h$ is injective in this
case.
\qed\medskip

\begin{lemma}\label{l.denjoy-continuity}
  The map $\tilde f$ defined in (\ref{e.h-def}) is strictly order-preserving and
  uniformly continuous on $A$.
\end{lemma}
\proof First, $\tilde f$ is strictly order-preserving on $A$ since this is true
for $h$, $h_{|A}^{-1}$ and $R$. In order to prove the uniform continuity, we
take arbitrary $x, x' \in A$ and look how their distance changes under the
action of $\tilde f$. We assume without loss of generality that
$d(x,x')=\Leb_X[x,x']$ ($x'$ is close to $x$ from the right) and use that
$h^*\Leb_X=\nu$ by Lemma~\ref{l.basic} to obtain
\begin{equation}\label{f-distance}
  d(x,x') \ = \ \nu[h(x),h(x')] \ = \ b \ d(h(x), h(x')) \ + \sum_{x_n \in [h(x), h(x')]} a_n \ .
\end{equation}
Now, since $h\circ f=R\circ h$ by definition, we obtain in the same way that 
\begin{equation}
  d(\tilde f(x),\tilde f(x')) \ =  \ b \ d(R(h(x)), R(h(x'))) \ + 
 \sum_{x_n \in [R(h(x)), R(h(x'))]} a_n \ .
\end{equation}
Since the rotation $R$ is an isometry and $x_n \in [R(h(x)), R(h(x'))]$ is equivalent to
$x_{n-1} \in [h(x), h(x')]$, we can rewrite the last line as
\begin{equation}
  d(\tilde f(x),\tilde f(x')) \ =  \ b \ d(h(x), h(x')) \ + \sum_{x_n \in [h(x), h(x')]} a_{n+1} \ .
\end{equation}
Thus
\begin{equation}\label{e.dist}
\begin{split}
d(\tilde f(x),\tilde f(x')) & = \  d(x,x') \ + \sum_{x_n \in [h(x), h(x')]} a_{n+1} - 
   \sum_{x_n \in [h(x), h(x')]} a_n\\
  & < \ \quad d(x,x') \ + \sum_{x_n \in [h(x), h(x')]} a_{n+1} \ . 
\end{split}
\end{equation}
Now fix $\eps>0$, choose $N \in \NNN$ such that $\sum_{|k| \geq N+1} a_k <
\varepsilon / 2$ and let 
\[
\delta \ = \  \min \{ \varepsilon / 2, \min \{a_k \sep |k|
\leq N+1 \} \} \ .
\]
Note that by (\ref{f-distance}) the fact that $d(x,x')<\delta$ implies that
$x_n\notin [h(x),h(x')]$ for all $n\leq N+1$.  Taking $x,x' \in A$ arbitrary
such that $d(x,x') < \delta$, the first term on the right side in (\ref{e.dist})
is obviously smaller than $\varepsilon / 2$, and the same is true for the second
one because from $x_n \in [h(x), h(x')]$ we get that $|n| > N+1$ and thus $|n+1| \geq
N+1$. Consequently, we obtain that $d(x,x')<\delta$ implies $d(\tilde f(x),\tilde f(x'))
< \eps$. Since $\delta$ did not depend on $x,x'$, this shows the uniform
continuity of $\tilde f$. \qed\medskip

\begin{lemma}
  The map $\tilde f$ defined in (\ref{e.denjoy-def}) can be extended to an
  orientation-preserving circle homeomorphism $f$ with wandering intervals
  $I_n$. Further, $f$ is semiconjugate to $R$. 
\end{lemma}
\proof In fact, there is almost nothing to prove anymore. The map $\tilde f$ is strictly
order-preserving of $A$, and by (\ref{e.cn_to_cn+1}) it maps the left endpoint
$c_n$ of the intervals $I_n$ to $c_{n+1}$. The right endpoints $d_n$ are not
contained in $A$ a priori, but if we extend $\tilde f$ continuously to the
closure of $A$ then by monotonicity and continuity $d_n$ has to be sent to
$d_{n+1}$ as well. Therefore, we may extend $\tilde f$ further by defining
$f_{|(c_n,d_n)} :(c_n,d_n)\to(c_{n+1},d_{n+1})$ in a more or less arbitrary way
(as long as it is monotone and continuous). 

The equation $h \circ \tilde f = R\circ h$ holds on $A$ by definition, and since
$h$ collapses the gaps $I_n=[c_n,d_n]$ to single points it extends to all of
\kreis, independently of the choice of $f_{|\inte(I_n)}$. Finally, since $f(I_n)=I_{n+1}$
for all $n\in\Z$ and all these intervals are disjoint, they are obviously
wandering. \qed

\section{Construction of almost automorphic minimal sets}
\label{BlowUp}

The following theorem is the core result of the paper. Due to the
particular attention paid to qpf systems, and also for the sake of
readability, we formulate this result specifically in this setting.
However, it can be straightforwardly generalized for a wider class of
driving forces, and this is done in the next section. We first
construct a special extension $\hat f$ of a given qpf system $f$ with
required properties. For understanding how the changed dynamics
exactly works, the crucial part of the proof is the one concerned with
the continuity of $\hat f$.  As a corollary of the construction, we
obtain an example of an almost automorphic minimal set which is not
filled-in.

We work with maps strictly monotone on the fibers. These may not be
homeomorphisms since they may fail to be surjective, but in this case
we can easily modify the system to get a homeomorphism without
affecting the original dynamics --- we simply extend the map linearly on
the fibers to a bigger annulus such that the resulting map is a
homeomorphism there.  Hence, for our purposes we can assume, without
loss of generality, our maps to be homeomorphisms.

\begin{theorem}\label{T:Main}
  Let $f : \AAA \to \AAA$ be a qpf increasing interval map
  and $\Gamma=\{(\theta,\gamma(\theta))\mid\theta\in\TTT^1\}$ an
  $f$-invariant curve with $\gamma:\TTT^1 \to \inter \III$.  Then
  there exists a topological extension $\hat f : \AAA \to \AAA$ of $f$
  with a factor map $h$ from $\hat f$ to $f$,
  $h(\theta,x)=(\theta,h_\theta(x))$, such that the following holds:
\romanlist
\item $\hat f$ is a qpf increasing interval map;
  \item all the fibre maps $h_\theta$ are non-decreasing;
  \item $h$ is injective on the complement of $h^{-1}(\Gamma)$;
  \item for a countable number of points in $\Gamma$ the preimage
    under $h$ is a vertical segment, for all other points in $\Gamma$
    it is a singleton;
  \item $h^{-1}(\Gamma)$ does not contain any graph of a continuous
    curve $\eta: \TTT^1 \to \III$;
  \item $h^{-1}(\Gamma)$ is pinched and contains an almost
    automorphic minimal set which is not almost periodic.
\listend
\end{theorem}
\begin{center}
\begin{figure} 
  \epsfig{file=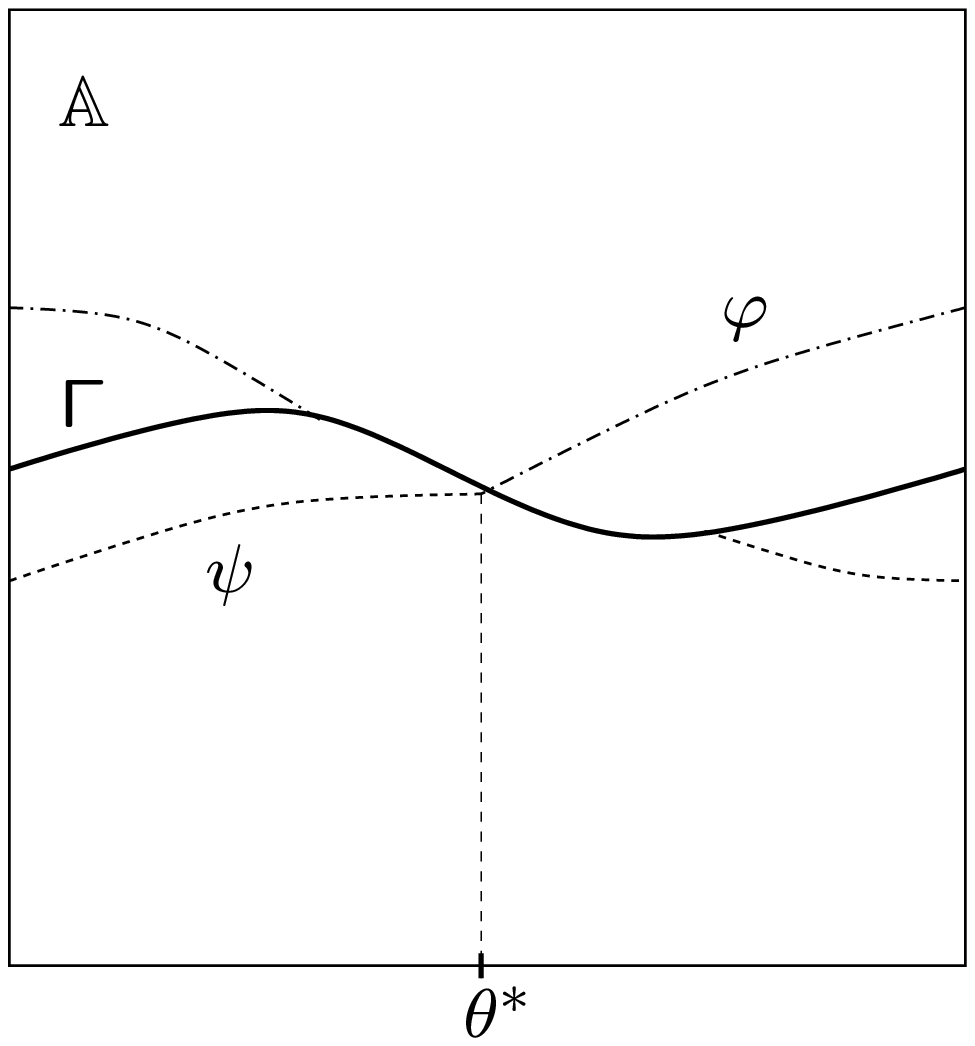, clip=, width=0.5\textwidth}
  \caption{\small Choice of the functions $\varphi$ and $\psi$.\label{f.mu}}
\end{figure}
\end{center}
\begin{proof}
  Denote by $R$ the driving irrational rotation of the system and put
  $\theta_{n}:=R^n(\theta)$ for $\theta\in\kreis$ and $n \in \Z$.
  Analogously to the previous section, fix a sequence $(a_n)_{n\in\ZZZ}$
  of strictly positive weights such that $\sum_{n \in \ZZZ} a_n < 1$
  and put $b := 1 - \sum_{n \in \ZZZ} a_n > 0$.  Fix an arbitrary
  $\theta^* \in \TTT^1$ and choose continuous functions $\varphi,
  \psi: \TTT^1 \to \III$ with the following properties (compare
  Figure~\ref{f.mu}):
\begin{enumerate}
\item $\psi \leq \gamma \leq \varphi$;
\item $\varphi=\gamma$ ($\psi=\gamma$) on a left (right) neighborhood
  of $\theta^*$;
\item $\psi < \varphi$ on $\TTT^1 \setminus \{\theta^*\}$.
\end{enumerate}
We will define a new measure $\mu$ on $\AAA$ via its fibre measures
$\mu_\theta$, that is, we first define a family
$(\mu_\theta)_{\theta\in\kreis}$ and then we let $\mu(A) :=
\int_{\TTT^1} \mu_\theta(A_\theta) d\theta$ for any measurable set $A
\subseteq \AAA$.  To that end, we first define a measure $\mu^0$ on
$\AAA$ via its fibre measures
\[
\mu_{\theta^*}^0 := \delta_{\gamma(\theta^*)} \qquad \mathrm{and} \qquad
\mu_\theta^0 := \frac{\Leb_{\vert
    [\psi(\theta),\phi(\theta)]}}{\varphi(\theta)-\psi(\theta)} \qquad \mathrm{for}
\; \theta \neq \theta^* \ .
\]
In particular, this means that $\mu^0_\theta$ varies continuously with
$\theta$ in the weak topology and it is non-atomic for
$\theta\neq\theta^*$.  Now define fiber measures of $\mu$ using
push-forwards of previous measures
\[
\mu_\theta \ := \ \sum_{n \in \ZZZ} a_n \mu^n_{\theta} + b \, \Leb_{\III} 
\]
where
\[
\mu^n_\theta \ := \ f^{n*}_{R^{-n}(\theta)} \mu^0_{R^{-n}(\theta)}\ = \
\mu^0_{R^{-n}(\theta)} f^{-n}_\theta \ .
\]
Note here that by convention
$(f_{R^{-n}(\theta)}^n)^{-1}=f_\theta^{-n}$.  It is easy to check that
the resulting measure $\mu$ can altervatively be defined by
\[
\mu := \sum_{n \in \ZZZ} a_n \mu^n + b \, \Leb_{\AAA} \qquad \textrm{where}
\qquad \mu^n := f^{n*} \mu^0 = \mu^0 f^{-n} \ .
\]

\noindent {\em Definition and continuity of $h$.}  Define the map $h:
\AAA \to \AAA$ by $h(\theta,x) = (\theta,h_\theta(x))$ where
\begin{equation}\label{e.hdef}
h_\theta(x) := \min \{ y \in \III \sep \mu_\theta[0,y] \geq x \} \ .
\end{equation}

To see the continuity of $h$ on $\AAA$ we can consider the sequence of
maps $h^{(k)}$, $k \in \N$, defined in completely analogous fashion to
the way how $h$ is defined via $\mu$, but this time with measures
$\mu^{(k)} := \sum_{m=-k}^k a_n \mu^n + \left( 1 - \sum_{m=-k}^k a_n
\right) \Leb_\AAA$.  All the maps $h^{(k)}$ are continuous due to the
construction of the $\mu^n$ -- note that the mappings $\theta\mapsto
\mu_\theta^n$ are continuous with respect to the weak topology -- and
they converge on $\A$ uniformly to $h$.\medskip

\noindent {\em Definition and continuity of $\hat f$.}
First, we define $\hat f$ in a natural way on a dense subset
$\Lambda\subseteq\A$ and then extend it to a continuous map on the
whole space by using uniform continuity on the subset.  

For all $\theta \notin \Orb(\theta^*)$, the measure $\mu_\theta$ is
non-atomic, hence $h_\theta$ is invertible by Lemma~\ref{l.basic}(ii).
Hence, $h$ is invertible on $\Lambda := (\TTT^1 \setminus
\Orb(\theta^*)) \times \III$.  Now define $\hat f: \Lambda \to
\Lambda$ by $\hat f := (h_{|\Lambda})^{-1} \circ f \circ
(h_{|\Lambda})$. We claim that $\hat f$ is uniformly continuous on
$\Lambda$ and thus extends to a continuous map on $\TTT^1 \times
\III$.

In order to show this, we first write what exactly $\hat f(\theta,x)$
does:
\begin{equation}\label{e.f}\begin{split}
  (\theta,x) & \overset{h}{\longmapsto} \ (\theta,h_\theta(x)) \ 
  \overset{f}{\longmapsto} \ 
  \left(R(\theta),f_\theta(h_\theta(x))\right) \\
  & \overset{h^{-1}}{\longmapsto} \ \left(R(\theta),h_{R(\theta)}^{-1}(f_\theta(h_\theta(x)))\right) 
  \ = \  \left(R(\theta),\mu_{R(\theta)}[0,f_\theta(h_\theta(x))]\right) 
\end{split}
\end{equation}
Since in the first variable we have an isometry, it is sufficient to
prove that $\hat f$ is uniformly continuous in the second coordinate.
Hence, we have to prove that for any $\eps>0$ there exists $\delta>0$
such that for any $\theta, \theta' \in \TTT^1$ and any $x,x' \in
\III$
\begin{equation}\label{e.preimage}
|x-x'| < \delta
\qquad \mathrm{and} \qquad d(\theta,\theta')<\delta
\end{equation}
implies
\begin{equation}\label{e.image}
  \left| \mu_{R(\theta)}[0,f_\theta(h_\theta(x))] - 
\mu_{R(\theta')}[0,f_{\theta'}(h_{\theta'}(x'))] \right| < \eps \ .
\end{equation}

To compare the expressions for the preimage and the image in (\ref{e.preimage})
and (\ref{e.image}), we realize that $x = h_\theta^{-1}(h_\theta(x)) =
\mu_\theta[0,h_\theta(x)]$ (since $h_\theta$ is invertible on each fibre
considered). In order to unfold (\ref{e.preimage}) and (\ref{e.image}) into the
respective series, according to the definition of $\mu_\theta$, we use the
abbreviation
\begin{equation}
  \label{eq:3}
  r_n \ := \ r_n(\theta,\theta',x,x') \ := \ \left( \mu^0_{R^{-n}(\theta)}
    f^{-n}_{\theta} [0,h_\theta(x)] - \mu^0_{R^{-n}(\theta')}
    f^{-n}_{\theta'} [0,h_{\theta'}(x')] \right) \ . 
\end{equation}
Note that due to the construction of $\mu^0$ the function $r_n$ depends
continuously on its arguments $\theta,\theta',x$ and $x'$ whenever $\theta\neq
R^n(\theta^*)\neq \theta'$. This will be crucial in the following.  We obtain
\begin{equation}
\begin{split}\label{e.unfold-preimage}
  P & := \ |x-x'|\ = \ \left| \mu_\theta[0,h_\theta(x)] -
    \mu_{\theta'}[0,h_{\theta'}(x')] \right|  \\
  & =\ \left| \sum_{n\in\Z} a_n r_n + b \, (h_\theta(x) -
    h_{\theta'}(x')) \right|
\end{split}
\end{equation}
and
\begin{equation}\begin{split}
\label{e.unfold-image}
I & := \ | \mu_{R(\theta)}[0,f_\theta(h_\theta(x))]  - \mu_{R(\theta')}[0,f_{\theta'}(h_{\theta'}(x'))] | \\
& = \ \left| \sum_{n\in\Z} a_n r_{n-1}  + b \,
(f_\theta(h_\theta(x)) - f_\theta'(h_{\theta'}(x'))) \right|
\end{split}
\end{equation}
where in \eqref{e.unfold-image} we used the identity
\begin{eqnarray*} 
  \mu^0_{R^{-n}(R(\theta))} f^{-n}_{R(\theta)} [0,f_\theta(h_\theta(x))] 
&=& \mu^0_{R^{-n+1}(\theta)} f^{-n+1}_{\theta} [0,h_\theta(x)] \ .
\end{eqnarray*}

Before providing precise estimates, let us now first describe the main
idea.  We have to show that if $P$ is small and $\theta$ and $\theta'$
are close then $I$ is small as well. Since $h$ and $f$ are uniformly
continuous on $\AAA$, we will get an easy control over the two
Lebesgue terms.  To compare the rest, we take into account that the
terms with big $|n|$ do not contribute much, since $\sum_{|n|\geq
  N}a_n \to 0$ as $N\to\infty$. Terms with small $|n|$ may contribute
more, but due to the continuity properties of the $r_n$ this is only
possible if $\theta$ or $\theta'$ are close to $R^n(\theta^*)$. Again
by assuming $\theta$ and $\theta'$ to be close, we can ensure that this
happens for at most one small integer $n^*$, so there is at most one
big term in \eqref{e.unfold-preimage}. However, since the sum $P$ is small,
this term cannot be too big either, hence in the end all the
contributions in \eqref{e.unfold-preimage} are small.  Finally, when going from
$P$ to $I$ all the terms are only multiplied by a factor $a_n/a_{n-1}$,
which leads to $I$ being small as well.

In order to give precise estimates, fix $\eps > 0$. Take $N_1 \in
\NNN$ such that
\begin{equation}\label{e.2}
  \sum_{|n|
    \geq N_1} a_n < \eps/{8} \ . 
\end{equation}
Let $a:=\max\{1, a_n/a_{n-1} \ssep |n| \leq N_1+1\}$ and choose
$\delta_1 \in(0,\eps/8)$ such that for $\delta\leq\delta_1$ in
(\ref{e.preimage}) the Lebesgue terms in (\ref{e.unfold-preimage}) and
\eqref{e.unfold-image} are always smaller than $\eps/8a$.  At the same
time, choose $N_2 \geq N_1$ such that
\begin{equation} \label{e.3} \sum_{|n| \geq N_2} a_n <  \eps/{8a} \ .
\end{equation}
Further, choose $\delta_2\in(0,\delta_1)$ such that
$d(\theta,\theta')<\delta_2$ implies the existence of at most one
$n^*\in\N$ with $|n^*|\leq N_2+1$ and
$\min\{d(\theta,R^{n^*-1}(\theta^*)),d(\theta',R^{n^*-1}(\theta^*))\}
< \delta_2$ --- eg. take $\delta_2\leq\halb\min_{k=1}^{2N_2+1}
  d(R^k(\theta^*),\theta^*))$.

Now, according to the remark made after \eqref{eq:3}, there exists
$\delta_3\in(0,\min\{\delta_2,\eps/8a\})$ with the property that if
$d(\theta,\theta')<\delta_3$ and both $\theta$ and $\theta'$ are
$\delta_3$-apart from $R^n(\theta^*)$, then both $a_nr_n$ and
$a_{n+1}r_{n}$ and are smaller than $\eps/16aN_2$.

In particular, the two last facts together imply that if
$\delta\leq\delta_3$ in (\ref{e.image}), then there exists at most one
$n^*\in\N$ with $|n^*|\leq N_2$ and $a_{n^*}r_{n^*-1}\geq \eps
/16aN_2$.  If $n^*\leq N_1$, then using the smallness of the Lebesgue
part of (\ref{e.image}) and (\ref{e.2}), we obtain
\begin{equation} \begin{split} \label{e.estimateI}
  I & \leq \  a_{n^*}r_{n^*-1}\ + \  \left| \sum_{|n|\leq N_1, n\neq n^*} a_n r_{n-1}\right| 
  \ +    \ 
  \left| \sum_{|n|> N_1} a_n r_{n-1}\right| \\ & + \ b|f_\theta h_\theta(x)-f_{\theta'}h_{\theta'}(x')|
  \  \leq \     a_{n^*}r_{n^*-1} \ + \ \frac{\eps}{2} \ . 
\end{split}
\end{equation}
Note that if $n^*>N_1$ or if $n^*$ with the above properties does not
exist, then we immediately obtain $I<\eps/2$. Otherwise, we obtain
from \eqref{e.unfold-preimage} and the fact that
$P<\delta<\delta_3<\eps/8a$ that
\begin{eqnarray*}
  a_{n^*-1}r_{n^*-1} & \leq & P + \left| \sum_{|n|\leq N_2, n\neq n^*-1} a_n r_{n}\right| 
  \ + \  \left| \sum_{|n|> N_2} a_n r_n\right| \\ & + & b|h_\theta(x)-h_{\theta'}(x')|
  \  \leq \     \frac{\eps}{2a} \ . 
\end{eqnarray*}
Consequently, since $|n^*|\leq N_1$ and thus $a_{n^*}/a_{n^*-1}\leq
a$, we have that $a_{n^*}r_{n^*-1}\leq \eps/2$. Plugging this into
(\ref{e.estimateI}) yields $I \leq \eps$. Thus, we obtain altogether
that $\delta<\delta_N$ implies $I\leq \eps$, as required.  \medskip

\noindent {\em Strict monotonicity of $\hat f$.}
The strict monotonicity of $\hat f$ is equivalent to its
invertibi\-li\-ty. In order to prove this, it suffices to show that $\hat
f^{-1}$ is uniformly continuous on $\Lambda$ as well, since it then
extends to a continuous function on the closure and provides an
inverse for $\hat f$ on all of $\AAA$. However, since $\hat
f^{-1}=h^{-1}\circ f^{-1} \circ h$ on $\Lambda$, this follows in
exactly the same way as the uniform continuity of $\hat f_{|\Lambda}$
in the preceeding step. The only difference is that $f$ is replaced by
$f^{-1}$.\medskip

The map $\hat f$ is obviously an extension of $f$ via $h$ since $h
\circ \hat f = f \circ h$ holds on $\Lambda$ by the definition of
$\hat f$ and it carries over to the closure of $\Lambda$ by
continuity. Thus, we have now constructed $h$ and $\hat f$ with the
properties stated in (i) and (ii). Properties (iii) and (iv) follow
easily from Lemma~\ref{l.basic}(i), since the fibre measures
$\mu_\theta$ have atoms placed exactly on the points
$\gamma(\theta^*_n)$, $n\in\ZZZ$. Property (vi) is a direct
consequence of (v), see Section~\ref{Preliminaries}, such that it only
remains to prove (v).  \medskip

 \noindent {\em Non-existence of continuous curves in
   $h^{-1}(\Gamma)$.} Recall that by Lemma~\ref{l.basic}(i) we have
 $$h^{-1}\{\gamma(\theta)\}\ =\
 [\mu_\theta[0,\gamma(\theta)),\mu_\theta[0,\gamma(\theta)]] \ . $$ We
 will show that
 \begin{equation}
   \label{e.discontinuity}
   \lim_{\theta \nearrow \theta^*}\mu_\theta[0,\gamma(\theta)) - 
   \lim_{\theta \searrow \theta^*}\mu_\theta[0,\gamma(\theta)] \ = \ a_0 \ .
 \end{equation}
 This implies immediately that any curve contained in $h^{-1}(\Gamma)$
 must have a discontinuity of size $a_0$ at $\theta^*$ (and hence, by
 invariance, discontinuities of size $a_n$ at all $\theta_n$).

 It follows from the definition of $\mu^0$ and the continuity of
 $\gamma$ that the mapping $\theta \mapsto
 \mu^0_\theta[0,\gamma(\theta)]$ is continuous on $\TTT^1 \setminus
 \{\theta^*\}$. Consequently, for the push-forwards $\mu^n$, $n\neq
 0$, the mappings $\theta \mapsto \mu^n_\theta[0,\gamma(\theta)]$ are
 continuous on $\TTT^1\setminus\{\theta^*_n\}$ and, in particular,
 have $\theta^*$ as a continuity point. Hence $\theta^*$ is a
 continuity point of $\theta \mapsto
 (\mu-a_0\mu^0)_\theta[0,\gamma(\theta)]$.  Therefore, we obtain
 \begin{eqnarray*}
   \lefteqn{\lim_{\theta \nearrow \theta^*}\mu_\theta[0,\gamma(\theta)) - 
   \lim_{\theta \searrow \theta^*}\mu_\theta[0,\gamma(\theta)] \ = }\\
 & = & \lim_{\theta \nearrow \theta^*}\mu^0_\theta[0,\gamma(\theta)) - 
 \lim_{\theta \searrow \theta^*}\mu^0_\theta[0,\gamma(\theta)] \ = \ a_0 
\end{eqnarray*}
 where the last equality follows from (2) and the definition of $\mu^0$.
\end{proof}

\begin{remark} \label{r.denjoy-dynamics} Note that the preimage of
  $\Gamma$ under $h$ in the above construction is homeomorphic to
  $\TTT^1$. A qualitative picture is given in Figure~\ref{f.preimage}
  . In order to make this precise, we represent $\TTT^1$ by $[0,1)$ to
  define the infimum of a subset. With this convention, we let
\[
\eta(\theta) = b\, \theta + \sum_{\theta_n \in [0,\theta)} a_n \qquad \qquad
\hat\eta(t) = \inf \{ \theta \sep \eta(\theta) \geq t \} \ .
\]
Note that thus $\eta$ is a right inverse of $\hat\eta$. Further, we
let $$\gamma^+(\theta) = \sup\{x\in\III\mid (\theta,x)\in h^{-1}(\Gamma)\}$$ and
finally
\[
\xi \ : \ \TTT^1 \to h^{-1}(\Gamma) \quad , \quad t \mapsto \left(\, \hat\eta(t) \, , \,
\gamma^+(\hat\eta(t))-(t-\eta(\hat\eta(t))\, \right).
\]
The fact that this provides a homeomorphism $\xi$ between $\TTT^1$ and
$h^{-1}(\Gamma)$ can now be checked easily.

Since the vertical segments in $h^{-1}(\Gamma)$ are wandering, this
means that $\xi$ is a conjugacy between $\hat f_{|h^{-1}(\Gamma)}$ and
a Denjoy counterexample on the circle. Using well-known results on
these maps, we obtain the following direct consequence of this remark.
\end{remark}

\begin{corollary}
Under the assumptions of the previous theorem,
$h^{-1}(\Gamma)$ contains exactly one minimal set.
This set is almost automorphic, not almost periodic, and not filled-in.
\end{corollary}

\begin{remark}
  We note that a slight modification of the above proof of
  Theorem~\ref{T:Main} also allows to produce examples where
  $h^{-1}(\Gamma)$ is filled-in. In order to do so, the functions
  $\varphi$ and $\psi$ only have to be chosen such that points
  $\theta$ with $\varphi(\theta)=\gamma(\theta)>\psi(\theta)$ and
  points $\theta'$ with $\varphi(\theta)>\gamma(\theta)=\psi(\theta)$
  both accumulate on $\theta^*$ from both sides. This will make the
  set $h^{-1}(\Gamma)$ `oscillate' close to $\theta^*$ in such a way,
  that the whole vertical segment over $\theta^*$ is contained in the
  closure of $h^{-1}(\Gamma)\smin (\{\theta^*\}\times\III)$.
\end{remark}

\begin{center}
\begin{figure}
  \epsfig{file=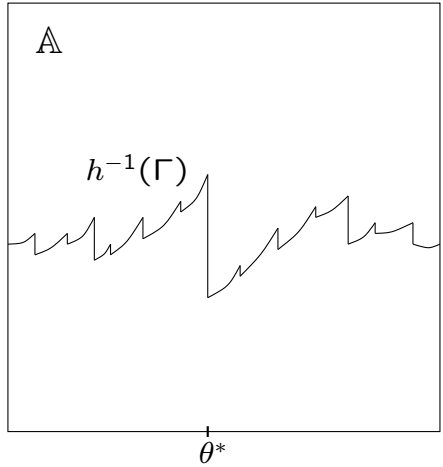, clip=, width=0.5\textwidth}
  \caption{\small A schematic picture of the set $h^{-1}(\Gamma)$. \label{f.preimage}}
\end{figure}
\end{center}

For the use in Section~\ref{Sharkovski}, we finally apply
Theorem~\ref{T:Main} to a quasiperiodically forced increasing
map $f$ which maps the annulus strictly inside itself and has a unique
continuous invariant curve $\Gamma$ as its global attractor, that is,
$\Gamma=\ncap f^n(\AAA)$. The resulting extension $\hat f$
then has no continuous invariant curves at all, and its global
attractor is the pinched set $h^{-1}(\Gamma)$. This leads to the
following
\begin{corollary}
  \label{c.nocurve}
  There exists a quasiperiodically forced increasing interval map
  $\hat f$ with the following properties:
  \begin{itemize}
  \item $\hat f$ has no continuous invariant curves;
  \item $\hat f(\AAA) \ssq \inte \AAA$;
  \item $\ncap \hat f^n(\AAA)$ is pinched and thus contains
    a unique almost automorphic minimal set $M$.
  \end{itemize}
\end{corollary}

\section{More general forcing}
\label{General}

The above construction can be adapted to more general forcing
processes.  Depending on what exactly we assume about the forcing, we
get different properties for the resulting system.

In fact, to obtain Theorem \ref{t.main} we only need to introduce
slight modifications in the definition of the functions $\varphi,
\psi: \Theta \to \III$ in the proof of Theorem~\ref{T:Main}. Since we
do not have one-dimensional structure and left and right neighborhoods
on $\Theta$ in general case, we have to change property (2) as
follows:
\begin{itemize}
\item[(2')] $\varphi_{|S} = \gamma_{|S}$ and $\psi_{|T} = \gamma_{|T}$ \ ,
\end{itemize}
where $S:=\{\sigma_n \in \Theta \ssep n \in \NNN\}$ and $T:=\{\tau_n
\in \Theta \ssep n \in \NNN\}$ are two pairwise disjoint sets with
$\sigma_n$ and $\tau_n$ converging to an arbitrarily chosen point
$\theta^* \in \Theta$. An explicit way how to define such functions
is 
\[
\varphi(\theta) := \gamma(\theta) + c \, \dist(\theta,S)
\qquad \mathrm{and} \qquad
\psi(\theta) := \gamma(\theta) - c \, \dist(\theta,T) \ ,
\]
where $c>0$ is a suitable scaling constant. 

In this situation, replacing the rotation $R$ by an almost periodic
minimal homeomorphism $\alpha$ of a compact metric space $\Theta$, we
can apply the proof of Theorem \ref{T:Main} without changes obtaining
all the properties (i)-(vi).  However, the minimal set obtained in the
end does not have to be non-filled-in, since the weaker condition (2')
does not suffice to guarantee this.

We can apply the construction even in more general situations
concerning the forcing $(\Theta, \alpha)$.  Let us assume that
$\Theta$ is an arbitrary compact metric space and $\alpha$ is a
homeomorphism of $\Theta$ with an aperiodic point $\theta^*$ which is
not isolated.  Still, all the properties from Theorem \ref{T:Main}
hold but the last one.  Even now, the set $h^{-1}(\Gamma)$ is pinched
and it is in fact an invariant closed set projecting onto the whole
driving space.  Minimality of $\alpha$ is only needed to ensure that
it contains a minimal set projecting onto $\Theta$, and almost
periodicity is only used to ensure that this minimal set is almost
automorphic.

Under the weakest assumptions mentioned, using the modified
construction above, we produce an example of an $\alpha$-forced
increasing interval map, a topological extension of the original one,
possesing an invariant closed set projecting on the whole driving
space such that it does not contain any graph of a continuous curve.

All the mentioned results can be also easily adapted for the case of
$\alpha$-forced circle homeomorphisms, ie. instead of increasing maps
on interval fibers we consider orientation preserving homeomorpisms on
circle fibers.

\section{A remark concerning a Sharkovsky-like theorem for qpf Maps}
\label{Sharkovski}

One of the most fundamental results in one-dimensional
dynamics is the Shar\-kov\-sky Theorem
\begin{theorem}[Sharkovsky, 1964]
  Suppose a continuous interval map $f$ has a periodic orbit of least
  period $n$. Then $f$ has periodic orbits of least period $m$ for all
  $m\in\NNN$ smaller than $n$ in the Sharkovsky ordering
\[
1 \triangleleft 2 \triangleleft 2^2 \triangleleft \ldots 2^2\cdot 7 
\triangleleft 2^2\cdot 5 \triangleleft 2^2 \cdot 3
\triangleleft\ldots 2\cdot 7 
\triangleleft 2\cdot 5 \triangleleft 2 \cdot 3
\triangleleft \ldots 7\triangleleft 5 \triangleleft 3 \ .
\]
\end{theorem}
In the context of skew product dynamics, it is a natural question to
ask whether this result can be generalised in a suitable way to forced
interval maps of the form~(\ref{e.skewproducts}). If the base dynamics
are aperiodic, an immediate problem that arises is to determine a
suitable class of objects that would play the role of periodic orbits
in this setting. A positive result in this direction given in
\cite{fabbri/jaeger/johnson/keller:2005} is based on the concept of
{\em core strips}. Unfortunately the definition of these objects is
rather technical, and we refrain from stating it here. More natural
analogues of periodic orbits would be closed curves of the form
$\Gamma=\{(\theta,\gamma(\theta)) \mid \theta \in \TTT^1\}$ with
continuous $\gamma:\TTT^1\to \III$.  However, the following simple
construction shows that, even in the quasiperiodic case, it is not
possible to obtain a Sharkovsky Theorem for forced interval maps by
replacing periodic points with periodic curves. It can therefore be
seen as a motivation for the use of the more complicated core strips
in \cite{fabbri/jaeger/johnson/keller:2005}.
\begin{proposition}
  There exists a quasiperiodically forced interval map which has a
  three-periodic continuous curve, but no invariant continuous curve.
 \end{proposition}
 \begin{proof}
  We start our construction with the direct product of an irrational
  rotation $R$ and an interval map $g$ such that 
  \begin{itemize}
  \item $g$ has a three-periodic orbit;
  \item $g$ has a unique fixed point $x_0$, which is repelling;
  \item $g$ is strictly increasing in a neighbourhood of $x_0$.
  \end{itemize}
  See Figure~\ref{f.g} for an example.
  
\begin{figure}\begin{center}
  \epsfig{file=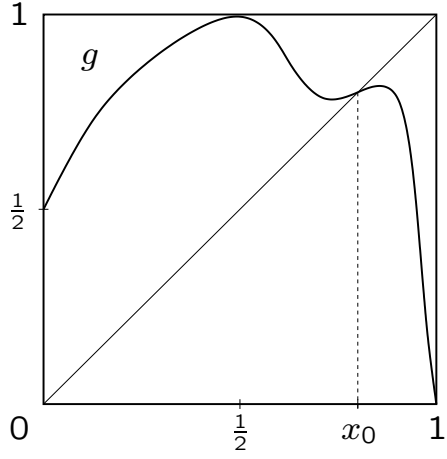, clip=, width=0.5\textwidth}
  \caption{\small The map $g$ with 3-periodic orbit $\{0,\halb,1\}$
    and unique and attracting fixed point $x_0$. \label{f.g}}\end{center}
\end{figure}

  Now, first note that if two periodic curves of $F=R\times g$
  intersect, then by minimality of $R$ they have to coincide. Hence,
  two periodic curves are either equal or disjoint. This implies in
  particular that all periodic curves of $F$ are constant, since for
  any periodic curve $\Gamma=\{(\theta,\gamma(\theta))\mid
  \theta\in\kreis\}$ and any $\rho\in\kreis$ the curve $\Gamma_\rho$
  parametrised by $\gamma_\rho: \theta\mapsto \gamma(\theta+\rho)$ is
  periodic as well and therefore must equal $\Gamma$. Hence, the
  periodic curves of $F$ correspond exactly to the periodic points of
  $g$, and in particular $F$ has a unique invariant curve
  $\Gamma_0=\kreis\times\{x_0\}$.

  Since $x_0$ is attracting, there exist $a_-,a_+\in\III$ such that
  $a_-<g(a_-)<x_0<g_(a_+)<a_+$ and $[a_-,a_+]$ does not contain any
  three-periodic point. Let $\AAA_0 := \kreis\times [a_-,a_+]$
  and $\widehat{\AAA}_0 := F(\AAA_0)=\kreis\times
  [g(a_-),g(a_+)]$. 

  Suppose now that $\hat f$ is a
  quasiperiodically forced increasing interval map with the properties
  given by Corollary~\ref{c.nocurve}. Let $\widehat{\AAA}:=\hat
  f(\AAA)\ssq \AAA$.  Choose a homeomorphism
  $h_1:\AAA\to\AAA_0$ of the form
  $h(\theta,x)=(\theta,h_\theta(x))$ which has increasing fibre maps
  $h_\theta$ and satisfies
  $h_1(\widehat{\AAA})=\widehat{\AAA}_0$.  Then define
  $\widehat F$ by
  \[
     \widehat F(\theta,x) \ = \ \left\{
       \begin{array}{cl}
         F(\theta,x)  & \textrm{if } (\theta,x)\notin \AAA_0\\
         \\
         h_1\circ \hat f \circ h_1^{-1} & \textrm{if } (\theta,x)\in\AAA_0
       \end{array}
\right.  \ .
\]
In other words, we replace the dynamics of $F$ on $\AAA_0$ by
those of $\hat f$, thus `destroying' the continuous invariant curve in
$\AAA_0$. Hence, the resulting map $\widehat F$ has no
continuous invariant curves anymore. However, it still has the
three-periodic invariant curves corresponding to the three-periodic
orbit of $g$, since these were not affected by the construction.
\end{proof}

\section{Reproduction of some examples by Rees} \label{Rees}

Given a homeomorphism $f:X\to X$ on a metric space $(X,d)$, a point
$x\in X$ is called {\em distal} if $\inf_{n\in\ZZZ} d(f^n(x),f^n(y)) > 0$
for any $y\in X$, $y\neq x$. A homeomorphism $f$ is called {\em point-distal} if there exists
a distal point, and it is called {\em distal} if all points in $X$ are
distal (as in the case of isometry).

In \cite{rees:1979}, Rees constructed point-distal, but non-distal
minimal homeomorphisms of the torus. Her examples are, in fact, skew
product transformations of the form $f:\TTT^2\to\TTT^2,\ (\theta,x)
\mapsto (\theta+\omega,f_\theta(x))$ with irrational
$\omega\in\TTT^1$. Furthermore, they are semi-conjugate to an
irrational rotation $R:(\theta,x) \mapsto (\theta+\omega,x+\rho)$ on
the torus, and the semi-conjugacy $h$ from $f$ to $R$ is bijective
except on a countable union of vertical segments which are mapped to a
single orbit of the rotation. The union of these segments equals the set of
non-distal points, all other points are distal. One may say that these
examples are obtained by {\em blowing up} the points of an orbit of
the rotation $R$ to vertical segments.

Although the techniques employed in Section~\ref{BlowUp} are quite
different from those used by Rees, they can easily be adapted in order
to reproduce such examples.

\begin{proposition}[Rees]
  There exists a point-distal, non-distal minimal transformation of the two-torus.
\end{proposition}
\begin{proof}[Sketch of the proof]
  We construct $\mu$ in a similar fashion as in Theorem \ref{T:Main}.
  Start with an irrational rotation $f$ of $\TTT^2$ and fix a point
  $z^*=(\theta^*,x^*)\in\TTT^2$. Without loss of generality we may
  assume that the orbit of $z^*$ does not intersect the line
  $\TTT^1\times\{0\}$.  Now cut the torus $\TTT^2$ along
  $\TTT^1\times\{0\}$ to get the annulus $\AAA$.  We define $h$ on
  $\AAA$ exactly as we did in the proof of Theorem \ref{T:Main} and
  then we go back to $\TTT^2$ by gluing $\AAA$ along
  $\TTT^1\times\{0\}$. Absence of atoms at $0$ for the fibre measures
  ensures continuity of the resulting torus map $h$.
  Denote by $f_1$ the rotation given by $f$ in the first coordinate.
  We define $\hat f: \Lambda \to \Lambda$ on
  $\Lambda:=\TTT^2\setminus(\Orb_{f_1}(\theta^*)\times\TTT^1)$
 by $\hat f := (h_{|\Lambda})^{-1} \circ f \circ
(h_{|\Lambda})$.
  In exactly analogous way as before we show uniform continuity of $\hat
  f$ on $\Lambda$ (and $\hat f^{-1}$ as well), thus $\hat f$ extends to a
  homeomorphism of the whole $\TTT^2$.

  It is easy to see now, as pointed out in the paragraphs above, that
  non-distal points of $\hat f$ are exactly those given by the
  $h$-preimages of the atoms of $\mu$ while the rest is formed by
  distal points of $\hat f$.
\end{proof}

\end{document}